\documentclass[12pt]{amsart}
\usepackage{amscd}
\usepackage{amsmath}
\usepackage{cancel}
\usepackage{graphicx}
\usepackage{hyperref}
\usepackage{tikz}
\usetikzlibrary{decorations.markings}
\tikzstyle{vertex}=[circle, draw, inner sep=0pt, minimum size=6pt]
\def\frk{\frak}               

\def\Phi{{\frk n}}
\def\Phi{{\frk N}}
\def\opn#1#2{\def#1{\operatorname{#2}}} 
\opn\chara{char} \opn\length{\ell} \opn\pd{pd} \opn\rk{rk}
\opn\projdim{proj\,dim} \opn\injdim{inj\,dim} \opn\rank{rank}
\opn\depth{depth} \opn\grade{grade} \opn\height{height}
\opn\embdim{emb\,dim} \opn\codim{codim}

\opn\Tr{Tr} \opn\bigrank{big\,rank}
\opn\superheight{superheight}\opn\lcm{lcm}
\opn\trdeg{tr\,deg}
\opn\reg{reg} \opn\lreg{lreg} \opn\ini{in} \opn\lpd{lpd}
\opn\size{size}
\opn\div{div} \opn\Div{Div} \opn\cl{cl} \opn\Cl{Cl}
\opn\Spec{Spec} \opn\Supp{Supp} \opn\supp{supp} \opn\Sing{Sing}
\opn\Ass{Ass} \opn\Min{Min}
\opn\Ann{Ann} \opn\Rad{Rad} \opn\Soc{Soc}
\opn\Im{Im} \opn\Ker{Ker} \opn\Coker{Coker} \opn\Am{Am}
\opn\Hom{Hom} \opn\Tor{Tor} \opn\Ext{Ext} \opn\End{End}
\opn\Aut{Aut} \opn\id{id}

\opn\nat{nat}
\opn\pff{pf}
\opn\Pf{Pf} \opn\GL{GL} \opn\SL{SL} \opn\mod{mod} \opn\ord{ord}
\opn\Gin{Gin} \opn\Hilb{Hilb}
\opn\aff{aff} \opn\con{conv} \opn\relint{relint} \opn\st{st}
\opn\lk{lk} \opn\cn{cn} \opn\core{core} \opn\vol{vol}
\opn\link{link} \opn\star{star}
\opn\gr{gr}

\def\pot#1#2{#1[\kern-0.28ex[#2]\kern-0.28ex]}

\opn\dirlim{\underrightarrow{\lim}}
\opn\inivlim{\underleftarrow{\lim}}
\def\Implies{\ifmmode\Longrightarrow \else
        \unskip${}\Longrightarrow{}$\ignorespaces\fi}
\def\implies{\ifmmode\Rightarrow \else
        \unskip${}\Rightarrow{}$\ignorespaces\fi}
\def\iff{\ifmmode\Longleftrightarrow \else
        \unskip${}\Longleftrightarrow{}$\ignorespaces\fi}

\let\:=\colon
\newcommand{\vertex}{\node[vertex]}
\newtheorem{Theorem}{Theorem}[section]
\newtheorem{Lemma}[Theorem]{Lemma}
\newtheorem{cor}[Theorem]{Corollary}
\newtheorem{Proposition}[Theorem]{Proposition}
\newtheorem{Remark}[Theorem]{Remark}

\newtheorem{Definition}[Theorem]{Definition}

\let\epsilon\varepsilon
\let\phi=\varphi
\let\kappa=\varkappa
\def\qed{\ifhmode\textqed\fi
      \ifmmode\ifinner\quad\qedsymbol\else\dispqed\fi\fi}
\def\textqed{\unskip\nobreak\penalty50
       \hskip2em\hbox{}\nobreak\hfil\qedsymbol
       \parfillskip=0pt \finalhyphendemerits=0}
\def\dispqed{\rlap{\qquad\qedsymbol}}

\opn\dis{dis}
\def\pnt{{\raise0.5mm\hbox{\large\bf.}}}

\opn\Lex{Lex}

\begin{document}
\title{On Algebraic Characterization of SSC of the Jahangir's Graph $\mathcal{J}_{n,m}$}

\author{Zahid Raza$^1$, Agha Kashif$^2$, Imran Anwar$^3$.}
 \thanks{ {\bf 1.} University of Sharjah, College of Sciences, Department of Mathematics, United Arab Emirates.
        \\{\bf 2.} University of Management and Technology, Lahore, Pakistan.
         \\ {\bf 3.} ASSMS, Government College University, Lahore, Pakistan.
         }
\email {zraza@sharjah.ac.ae, kashif.khan@umt.edu.pk,aghakashifkhan@hotmail.com, iimrananwar@gmail.com
 }

 \maketitle
\begin{abstract}
In this paper, some algebraic and combinatorial characterizations of the spanning simplicial complex $\Delta_s(\mathcal{J}_{n,m})$ of  the Jahangir's graph $\mathcal{J}_{n,m}$ are explored. We show that $\Delta_s(\mathcal{J}_{n,m})$ is pure, present the formula for $f$-vectors associated to it and hence deduce a recipe for computing the Hilbert series of the Face ring $k[\Delta_s(\mathcal{J}_{n,m})]$. Finaly, we show that the face ring of $\Delta_s(\mathcal{J}_{n,m})$ is Cohen-Macaulay and give some open scopes of the current work.

\end{abstract}
\section{introduction}

The concept of spanning simplicial complex (SSC) associated with the edge set of a simple finite connected graph is introduced by Anwar, Raza and Kashif in \cite{ARK}. They revealed some important algebraic properties of SSC of a unicyclic graph. Kashif, Raza and Anwar further established the theory and explored algebraic characterizations of some more general classes of n-cyclic graphs in \cite{KRA1, KRA2}. The problem of finding the SSC for a general simple finite connected graph is not an easy task to handel. Recently in \cite{ZSG} Zhu, Shi and Geng discussed the SSC of another class $n-$cyclic graphs with a common edge.
\\In this article, we discuss some algebraic and combinatorial properties of the spanning simplicial complex $\Delta_s(\mathcal{J}_{n,m})$ of a certain class of {\em cyclic graphs, $\mathcal{J}_{n,m}$}. For simplicity, we fixed $n=2$ in our results. Here, $\mathcal{J}_{n,m}$ is the class of Jahangir's graph defined in \cite{LJM} as follows:
\\{\em The Jahangir's graph $J_{n,m}$, for $m \ge 3$, is a graph on $nm +1$ vertices i.e., a graph consisting of a cycle $C_{nm}$ with one additional vertex which is adjacent to $m$ vertices of $C_{nm}$ at distance $n$ to each other on $C_{nm}$.}
\\More explicitly, it consists of a cycle $C_{nm}$ which is further divided into $m$ consecutive cycles $C_i$ of equal length such that all these cycles have one vertex common and every pair of consecutive cycles has exactly one edge common. For example the graph $\mathcal{J}_{2,3}$ is given as:\\
\[\begin{tikzpicture}

	\vertex[fill] (1) at (0,0)  {};
	\vertex[fill] (2) at (0,-1)  {};
	\vertex[fill] (3) at (0.8667,.5)  {};
	\vertex[fill] (4) at (-0.8667,.5)  {};
    \vertex[fill] (5) at (0,1)  {};
	\vertex[fill] (6) at (-0.8667,-.5)  {};
	\vertex[fill] (7) at (0.8667,-.5)  {};

\draw (30:0cm) -- node[below] {$e_{21}$} (30:1cm)
arc  (30:60:1cm) node[above]{$e_{13}$} arc (60:90:1cm)
arc  (90:120:1cm) node[above]{$e_{12}$} arc (120:150:1cm)
arc  (150:180:1cm) node[left]{$e_{33}$} arc (180:210:1cm)
arc  (210:240:1cm) node[left]{$e_{32}$} arc (240:270:1cm)
arc  (270:300:1cm) node[right]{$e_{23}$} arc (300:330:1cm)
arc  (330:360:1cm) node[right]{$e_{22}$} arc (0:30:1cm)
(150:0cm) -- node[above] {$e_{11}$} (150:1cm)
(270:0cm) -- node[left] {$e_{31}$} (270:1cm)
;
\end{tikzpicture}\]
\begin{center}
The graph $\mathcal{J}_{2,3}$\\
\tiny{Figure 1}
\end{center}

 We fix the edge set of $\mathcal{J}_{2,m}$ as follows:
 \begin{equation}
 E=\{e_{11},e_{12},e_{13},e_{21},e_{22},e_{23},\hdots,e_{m1},e_{m2},e_{m3}\}.
 \end{equation}
 Here, $\{e_{k1},e_{k2},e_{k3},e_{(k+1)1}\}$ is the edge set of the cycle $C_k$ for $k\in\{1,2,\hdots,m-1\}$ and $\{e_{m1},e_{m2},e_{m3},e_{11}\}$ is the edge set of cycle $C_m$.  Also $e_{k1}$ always represents the common edge between $C_{k-1}$ and $C_k$ for $k\in\{1,2,\hdots,m-1\}$ and $e_{11}$ is the common edge between the cycle $C_m$ and $C_1$.

\section{Preliminaries}
In this section, we give some background and preliminaries of the topic and define some important notions to make this paper self-contained. However, for more details of the notions we refer the reader to \cite{BH, F1, F2, Ha, HH, MS, Vi}.

\begin{Definition}\label{spa}
\em {A {\em spanning tree} of a simple connected finite graph $G(V,E)$ is
a subtree of $G$ that contains every vertex of $G$.

We represent the collection of all edge-sets of the spanning trees
of $G$ by $s(G)$, in other words;
$$s(G):=\{E(T_i)\subset E ,    \hbox {\, where $T_i$ is a spanning tree of $G$}\}.$$
}
\end{Definition}
We can obtain the spanning tree of the Jahangir's graph $\mathcal{J}_{2,m}$ by removing exactly $m$ edges from it keeping in view the following:\\
\begin{itemize}
  \item Not more than one edge can be removed from the non common edges of any cycle.
  \item If a common edge between two or more consecutive cycles is removed then exactly one edge must be removed from the resulting big cycle.
  \item Not all common edges can be removed simultaneously.
\end{itemize}
This method is referred to as the {\em cutting-down method}. For example, by using the {\em cutting-down method} for the graph $\mathcal{J}_{2,3}$ given
in Fig. 1 we obtain:

$s(\mathcal{J}_{2,3})=\big\{ \{ e_{11}, e_{21}, e_{31}, e_{12}, e_{22}, e_{32}\},\{ e_{11}, e_{21}, e_{31}, e_{12}, e_{22}, e_{33}\},\{ e_{11}, e_{21}, e_{31}, e_{12}, e_{23}, \\
e_{32}\},\{ e_{11}, e_{21}, e_{31}, e_{12}, e_{23}, e_{33}\},\{ e_{11}, e_{21}, e_{31}, e_{13}, e_{22}, e_{32}\},\{ e_{11}, e_{21}, e_{31}, e_{13}, e_{22}, e_{33}\},\{ e_{11},\\
e_{21}, e_{31}, e_{13}, e_{23}, e_{32}\},\{ e_{11}, e_{21}, e_{31}, e_{13}, e_{23}, e_{33}\},\{ e_{21}, e_{31}, e_{32}, e_{33}, e_{12}, e_{22}\},\{ e_{21}, e_{31}, e_{32},e_{33},\\
e_{12}, e_{23}\},\{ e_{21}, e_{31}, e_{32}, e_{33}, e_{13}, e_{22}\},\{ e_{21}, e_{31}, e_{32}, e_{33}, e_{13}, e_{23}\},\{ e_{21}, e_{31}, e_{12}, e_{13}, e_{33}, e_{22}\},\\
\{ e_{21}, e_{31}, e_{12}, e_{13}, e_{33}, e_{23}\},\{ e_{21}, e_{31}, e_{12}, e_{13}, e_{32}, e_{22}\},\{ e_{21}, e_{31}, e_{12}, e_{13}, e_{32}, e_{23}\},\{ e_{11}, e_{31}, e_{12},\\
 e_{13}, e_{22}, e_{32}\},\{ e_{11}, e_{31}, e_{12}, e_{13}, e_{22}, e_{33}\},\{ e_{11}, e_{31}, e_{12},e_{13}, e_{23}, e_{32}\},\{ e_{11}, e_{31}, e_{12},e_{13}, e_{23}, e_{33}\},\\
 \{ e_{11}, e_{31}, e_{22}, e_{23}, e_{13}, e_{32}\},\{ e_{11}, e_{31}, e_{22}, e_{23}, e_{13}, e_{33}\},\{ e_{11}, e_{31}, e_{22}, e_{23}, e_{12}, e_{32}\},\{ e_{11}, e_{31}, e_{22},\\
 e_{23}, e_{12}, e_{33}\},\{ e_{11}, e_{21}, e_{23}, e_{22}, e_{32}, e_{12}\},\{ e_{11}, e_{21}, e_{23}, e_{22}, e_{32}, e_{13}\},\{ e_{11}, e_{21}, e_{23}, e_{22}, e_{33}, e_{12}\},\\
 \{ e_{11}, e_{21}, e_{23}, e_{22}, e_{33}, e_{13}\},\{ e_{11}, e_{21}, e_{32}, e_{33}, e_{22}, e_{12}\},\{ e_{11}, e_{21}, e_{32}, e_{33}, e_{22}, e_{13}\},\{ e_{11}, e_{21}, e_{32},\\
 e_{33}, e_{23}, e_{12}\},\{ e_{11}, e_{21}, e_{32}, e_{33}, e_{23}, e_{13}\},\{ e_{11}, e_{13}, e_{22}, e_{23}, e_{32}, e_{33}\},\{ e_{11}, e_{12}, e_{22}, e_{23}, e_{32}, e_{33}\},\\
 \{ e_{11}, e_{12}, e_{13}, e_{23}, e_{32}, e_{33}\},\{ e_{11}, e_{12}, e_{13}, e_{22}, e_{32}, e_{33}\},\{ e_{11}, e_{12}, e_{13}, e_{22}, e_{23}, e_{33}\},\{ e_{11}, e_{12}, e_{13},\\
  e_{22}, e_{23}, e_{32}\},\{ e_{21}, e_{13}, e_{22}, e_{23}, e_{32}, e_{33}\},\{ e_{21}, e_{12}, e_{22}, e_{23}, e_{32}, e_{33}\}, \{ e_{21}, e_{12}, e_{13}, e_{23}, e_{32}, e_{33}\},\\
  \{ e_{21}, e_{12}, e_{13}, e_{22}, e_{32}, e_{33}\},\{ e_{21}, e_{12}, e_{13}, e_{22}, e_{23}, e_{33}\},\{ e_{21}, e_{12}, e_{13}, e_{22}, e_{23}, e_{32}\},
  \{ e_{31}, e_{13}, e_{22},\\
   e_{23}, e_{32}, e_{33}\},\{ e_{31}, e_{12}, e_{22}, e_{23}, e_{32}, e_{33}\}, \{ e_{31}, e_{12}, e_{13}, e_{23}, e_{32}, e_{33}\},\{ e_{31}, e_{12}, e_{13}, e_{22}, e_{32}, e_{33}\},\\
   \{ e_{31}, e_{12}, e_{13}, e_{22}, e_{23}, e_{33}\},\{ e_{31}, e_{12}, e_{13}, e_{22}, e_{23}, e_{32}\}
  \big\}.$

\[\begin{tikzpicture}

	\vertex[fill] (1) at (0,0)  {};
	\vertex[fill] (2) at (0,-2)  {};
	\vertex[fill] (3) at (1.732,1)  {};
	\vertex[fill] (4) at (-1.732,1)  {};
    \vertex[fill] (5) at (0,2)  {};
	\vertex[fill] (6) at (-1.732,-1)  {};
	\vertex[fill] (7) at (1.732,-1)  {};

%
%

\draw (30:0cm) -- node[below] {$e_{21}$} (30:2cm)
arc  (30:60:2cm) node[above]{$e_{13}$} arc (60:90:2cm)
arc  (90:120:2cm) node[above]{$e_{12}$} arc (120:150:2cm)
arc  (150:180:2cm) node[left]{$e_{33}$} arc (180:210:2cm)
arc  (210:240:2cm) node[left]{$e_{32}$} arc (240:270:2cm)
arc  (270:300:2cm) node[right]{$e_{23}$} arc (300:330:2cm)
arc  (330:360:2cm) node[right]{$e_{22}$} arc (0:30:2cm)
(150:0cm) -- node[above] {$e_{11}$} (150:2cm)
(270:0cm) -- node[left] {$e_{31}$} (270:2cm)
;
\end{tikzpicture}\]
\begin{center}
The graph $\mathcal{J}_{2,3}$\\
\tiny{Figure 1}
\end{center}

\begin{Definition}{\em
A {\em simplicial complex} $\Delta$ over a finite set
$[n]=\{1, 2,\ldots,n \}$ is a collection of subsets of $[n]$, with
the property that $\{i\}\in \Delta$ for all $i\in[n]$, and if $F\in
\Delta$  then $\Delta$ will contain all the subsets of $F$
(including the empty set). An element of $\Delta$ is called a face
of $\Delta$, and the dimension of a face $F$ of $\Delta$ is defined
as $|F|-1$, where $|F|$ is the number of vertices of $F$. The
maximal faces of $\Delta$ under inclusion are called facets of
$\Delta$. The dimension of the simplicial complex $\Delta$ is :
$$\hbox{dim} \Delta = \max\{\hbox{dim} F | F \in \Delta\}.$$
We denote the simplicial complex $\Delta$ with facets $\{F_1,\ldots
, F_q\}$ by $$\Delta = \big\langle F_1,\ldots, F_q\big\rangle $$ }
\end{Definition}
\begin{Definition}\label{fvec}
{\em
For a simplicial complex $\Delta$ having dimension $d$, its
$f-vector$ is a $d+1$-tuple, defined as:
$$f(\Delta)=(f_0,f_1,\ldots,f_d)$$
where $f_i$ denotes the number of $i-dimensional$ faces of $\Delta.$
}\end{Definition}

\begin{Definition}\label{ssc}{\bf (Spanning Simplicial Complex )}\\
{\em
Let $G(V,E)$ be a a simple finite connected graph and $s(G)=\{E_1,
E_2,\ldots,E_t\}$ be the edge-set of all possible spanning trees of
$G(V,E)$, then we defined (in \cite{ARK}) a simplicial complex $\Delta_s(G)$ on $E$ such
that the facets of $\Delta_s(G)$ are precisely the elements of
$s(G)$, we call $\Delta_s(G)$ as the {\em spanning simplicial
complex} of $G(V,E)$. In other words;
$$\Delta_s(G)=\big\langle E_1,E_2,\ldots,E_t\big\rangle.$$
}\end{Definition}
For example; the spanning simplicial complex of
the graph $\mathcal{J}_{2,3}$ given in Fig. 1 is:

$\Delta_s(\mathcal{J}_{2,3})=\big\langle \{ e_{11}, e_{21}, e_{31}, e_{12}, e_{22}, e_{32}\},\{ e_{11}, e_{21}, e_{31}, e_{12}, e_{22}, e_{33}\},\{ e_{11}, e_{21}, e_{31}, e_{12}, e_{23}, \\
e_{32}\},\{ e_{11}, e_{21}, e_{31}, e_{12}, e_{23}, e_{33}\},\{ e_{11}, e_{21}, e_{31}, e_{13}, e_{22}, e_{32}\},\{ e_{11}, e_{21}, e_{31}, e_{13}, e_{22}, e_{33}\},\{ e_{11},\\
e_{21}, e_{31}, e_{13}, e_{23}, e_{32}\},\{ e_{11}, e_{21}, e_{31}, e_{13}, e_{23}, e_{33}\},\{ e_{21}, e_{31}, e_{32}, e_{33}, e_{12}, e_{22}\},\{ e_{21}, e_{31}, e_{32},e_{33},\\
e_{12}, e_{23}\},\{ e_{21}, e_{31}, e_{32}, e_{33}, e_{13}, e_{22}\},\{ e_{21}, e_{31}, e_{32}, e_{33}, e_{13}, e_{23}\},\{ e_{21}, e_{31}, e_{12}, e_{13}, e_{33}, e_{22}\},\\
\{ e_{21}, e_{31}, e_{12}, e_{13}, e_{33}, e_{23}\},\{ e_{21}, e_{31}, e_{12}, e_{13}, e_{32}, e_{22}\},\{ e_{21}, e_{31}, e_{12}, e_{13}, e_{32}, e_{23}\},\{ e_{11}, e_{31}, e_{12},\\
 e_{13}, e_{22}, e_{32}\},\{ e_{11}, e_{31}, e_{12}, e_{13}, e_{22}, e_{33}\},\{ e_{11}, e_{31}, e_{12},e_{13}, e_{23}, e_{32}\},\{ e_{11}, e_{31}, e_{12},e_{13}, e_{23}, e_{33}\},\\
 \{ e_{11}, e_{31}, e_{22}, e_{23}, e_{13}, e_{32}\},\{ e_{11}, e_{31}, e_{22}, e_{23}, e_{13}, e_{33}\},\{ e_{11}, e_{31}, e_{22}, e_{23}, e_{12}, e_{32}\},\{ e_{11}, e_{31}, e_{22},\\
 e_{23}, e_{12}, e_{33}\},\{ e_{11}, e_{21}, e_{23}, e_{22}, e_{32}, e_{12}\},\{ e_{11}, e_{21}, e_{23}, e_{22}, e_{32}, e_{13}\},\{ e_{11}, e_{21}, e_{23}, e_{22}, e_{33}, e_{12}\},\\
 \{ e_{11}, e_{21}, e_{23}, e_{22}, e_{33}, e_{13}\},\{ e_{11}, e_{21}, e_{32}, e_{33}, e_{22}, e_{12}\},\{ e_{11}, e_{21}, e_{32}, e_{33}, e_{22}, e_{13}\},\{ e_{11}, e_{21}, e_{32},\\
 e_{33}, e_{23}, e_{12}\},\{ e_{11}, e_{21}, e_{32}, e_{33}, e_{23}, e_{13}\},\{ e_{11}, e_{13}, e_{22}, e_{23}, e_{32}, e_{33}\},\{ e_{11}, e_{12}, e_{22}, e_{23}, e_{32}, e_{33}\},\\
 \{ e_{11}, e_{12}, e_{13}, e_{23}, e_{32}, e_{33}\},\{ e_{11}, e_{12}, e_{13}, e_{22}, e_{32}, e_{33}\},\{ e_{11}, e_{12}, e_{13}, e_{22}, e_{23}, e_{33}\},\{ e_{11}, e_{12}, e_{13},\\
  e_{22}, e_{23}, e_{32}\},\{ e_{21}, e_{13}, e_{22}, e_{23}, e_{32}, e_{33}\},\{ e_{21}, e_{12}, e_{22}, e_{23}, e_{32}, e_{33}\}, \{ e_{21}, e_{12}, e_{13}, e_{23}, e_{32}, e_{33}\},\\
  \{ e_{21}, e_{12}, e_{13}, e_{22}, e_{32}, e_{33}\},\{ e_{21}, e_{12}, e_{13}, e_{22}, e_{23}, e_{33}\},\{ e_{21}, e_{12}, e_{13}, e_{22}, e_{23}, e_{32}\},
  \{ e_{31}, e_{13}, e_{22},\\
   e_{23}, e_{32}, e_{33}\},\{ e_{31}, e_{12}, e_{22}, e_{23}, e_{32}, e_{33}\}, \{ e_{31}, e_{12}, e_{13}, e_{23}, e_{32}, e_{33}\},\{ e_{31}, e_{12}, e_{13}, e_{22}, e_{32}, e_{33}\},\\
   \{ e_{31}, e_{12}, e_{13}, e_{22}, e_{23}, e_{33}\},\{ e_{31}, e_{12}, e_{13}, e_{22}, e_{23}, e_{32}\}\big\rangle.$

\section{Spanning trees of $\mathcal{J}_{2,m}$ and Face ring $\Delta_s(\mathcal{J}_{2,m})$ }

In this section, we give two lemmas which give some important characterization of the graph $\mathcal{J}_{2,m}$ and its spanning simplicial complex  $s(\mathcal{J}_{2,m})$. We present a proposition which gives the $f$-vectors and dimension of the $\mathcal{J}_{2,m}$. Finally, in Theorem \ref{Hil} we give the formulation for the Hilbert series of the Face ring $k\big[\Delta_s(\mathcal{J}_{2,m})\big]$.

\begin{Lemma}\label{lema1}{\bf Characterization of $\mathcal{J}_{2,m}$}\\
\em{Let $\mathcal{J}_{2,m}$ be the graph with the edges $E$ as defined in eq. (1) and $C_1,C_2,\cdots,C_m$ be its $m$ consecutive cycles of equal lengths, then the total number of cycles in the graph are
$$\tau=m^2$$
such that $\Big|C_{i_1,i_2,\hdots,i_k}\Big|=2(k+1).$ }
\end{Lemma}

\begin{proof}
The Jhangir's graph $\mathcal{J}_{2,m}$ contains more than just $m$ consecutive cycles. The remaining cycles can be obtained by deleting the common edges between any number (but all) of consecutive cycles and getting a cycle by their remaining edges. The cycle obtained in this way by adjoining consecutive cycles $C_{i_1},C_{i_2},\hdots,C_{i_k}$ is denoted by $C_{i_1,i_2,\hdots,i_k}$.Therefore, we get the following cycles\\
$C_{1,2},C_{2,3},\hdots,C_{m-1,m},C_{m,1},C_{1,2,3},\hdots,C_{m-2,m-1,m},C_{m-1,m,1},C_{m,1,2},\hdots,C_{1,2,3,\hdots,m},\\
C_{2,3,4,\hdots,m,1},C_{3,4,5,\hdots,m,1,2},C_{m,1,2,\hdots,m-1}.$\\

Combining these with given $m$ cycles we have total cycles in the graph $\mathcal{J}_{2,m}$,
$$C_{i_1,i_2,\hdots,i_k}\;\;\;i_j\in\{1,2,\hdots,m\}\; and \; 1 \le k\le m,$$
such that $i_{j+1}=i_j+1$ if $i_j\neq m$ and $i_{j+1}=1$ if $i_j= m.$ \\
Now for a fixed value of $k$, simple counting reveals that the total number of cycles $C_{i_1,i_2,\hdots,i_k}$ are $m$. Hence the total number of cycles in $\mathcal{J}_{2,m}$ is $\tau$. Also it is clear from the construction above that $C_{i_1,i_2,\hdots,i_k}$ is obtained by deleting common edges between consecutive cycles $C_{i_1},C_{i_2},\hdots,C_{i_k}$ which are $k-1$ in number. Therefore, the order of the cycle $C_{i_1,i_2,\hdots,i_k}$ is obtained by adding orders of all $C_{i_1},C_{i_2},\hdots,C_{i_k}$ and subtracting $2(k-1)$ from it, since the common edges are being counted twice in sum. This implies
$$ \Big|C_{i_1,i_2,\hdots,i_k}\Big|=\sum\limits_{t=1}^{k}\big|C_{i_t}\big|-2(k-1)=2(k+1).$$

We denote $\beta_{i_1,i_2,\hdots,i_k}=\Big|C_{i_1,i_2,\hdots,i_k}\Big|$.\\
\end{proof}
In the following results, we fix $C_{u_1,u_2,\hdots,u_p},\; C_{v_1,v_2,\hdots,v_q}$ to represent any two cycles from the cycles $$C_{i_1,i_2,\hdots,i_k}\;\;\;i_j\in\{1,2,\hdots,m\}\; and \; 1 \le k\le m,$$
such that $i_{j+1}=i_j+1$ if $i_j\neq m$ and $i_{j+1}=1$ if $i_j= m,$ of the graph $\mathcal{J}_{2,m}$. Also we fix the notation "$a\rightarrow b$" if $b$ immediately proceeds $a$ i.e., the very next in order of preferences.

\begin{Proposition}\label{prp2}
\em{Let $\mathcal{J}_{2,m}$ be the graph with the edges $E $ as defined in eq. (1) such that $\{u_1,u_2,\hdots,u_p\}\subseteq \{v_1,v_2,\hdots,v_q\}$ then we have
$$\Big|C_{u_1,u_2,\hdots,u_p}\bigcap C_{v_1,v_2,\hdots,v_q}\Big|=\left\{
                                                       \begin{array}{ll}
                                                         \beta_{u_1,u_2,\hdots,u_p}-2, & {\{u_1,u_p\}\not\subseteq\; \{v_1,v_q\}} \\
                                                         \beta_{u_1,u_2,\hdots,u_p}-1, & {u_1 \in \{v_1,v_q\}\;\&\;u_p\notin\{v_1,v_q\}} \\
                                                         \beta_{u_1,u_2,\hdots,u_p}-1, & {u_p \in \{v_1,v_q\}\;\&\;u_1\notin\{v_1,v_q\}} \\
                                                         \beta_{u_1,u_2,\hdots,u_p},   & {u_1=v_1\; \&\; u_p=v_q\; or\; u_1=v_q\;\; u_p=v_q}
                                                       \end{array}
                                                     \right.$$}
\end{Proposition}
\begin{proof}
Since the cycles $C_{u_1,u_2,\hdots,u_p}$ and $C_{v_1,v_2,\hdots,v_q}$ are obtained by deleting the common edges between cycles $C_{u_1},C_{u_2},\hdots,C_{u_p}$ and $C_{v_1},C_{v_2},\hdots,C_{v_q}$ respectively. Therefore, ${\{u_1,u_p\}\not\subseteq\; \{v_1,v_q\}}$ implies $\{u_1,u_2,\hdots,u_p\}\subset \{v_1,v_2,\hdots,v_q\}$. Hence the intersection $C_{u_1,u_2,\hdots,u_p}\bigcap C_{v_1,v_2,\hdots,v_q}$ will contain only the non common edges of the cycle $C_{u_1,u_2,\hdots,u_p}$ excluding its two edges common with the cycles on its each end. This gives the order of intersection in this case as $\beta_{u_1,u_2,\hdots,u_p}-2$. The remaining cases can be visualized in similar manner.
\end{proof}

\begin{Proposition}\label{prp3}
\em{Let $\mathcal{J}_{2,m}$ be the graph with the edges $E $ as defined in eq. (1) such that $\{\overline{u}_1,\overline{u}_2,\hdots,\overline{u}_\sigma\}\subseteq \{v_1,v_2,\hdots,v_q\}$ and $\overline{u}_t\in\{u_1,u_2,\hdots,u_p\}\;\&\;\overline{u}_{t-1}\rightarrow\overline{u}_t$  with $t\le\sigma<p$ then we have

$$\Big|C_{u_1,u_2,\hdots,u_p}\bigcap C_{v_1,v_2,\hdots,v_q}\Big|=\left\{
                                                       \begin{array}{ll}
                                                         \beta_{\overline{u}_1,\overline{u}_2,\hdots,\overline{u}_\sigma}-1, & {\overline{u}_1=v_1\;\&\;v_q\rightarrow u_1} \\
                                                         \beta_{\overline{u}_1,\overline{u}_2,\hdots,\overline{u}_\sigma}-2, & {\overline{u}_1=v_1\;\&\;v_q\not\rightarrow u_1} \\
                                                         \beta_{\overline{u}_1,\overline{u}_2,\hdots,\overline{u}_\sigma}-1, & {\overline{u}_\sigma=v_q\;\&\;u_p\rightarrow v_1} \\
                                                         \beta_{\overline{u}_1,\overline{u}_2,\hdots,\overline{u}_\sigma}-2, & {\overline{u}_\sigma=v_q\;\&\;u_p\not\rightarrow v_1} \\
                                                       \end{array}
                                                     \right.$$}
\end{Proposition}

\begin{proof}
Here the cycles $C_{\overline{u}_1},C_{\overline{u}_2},\hdots,C_{\overline{u}_\sigma}$ are amongst $\sigma$ consecutive adjoining cycles of the cycle $C_{u_1,u_2,\hdots,u_p}$ which are also overlapping with the $\sigma$ consecutive adjoining cycles of the cycle $C_{v_1,v_2,\hdots,v_q}$. If the adjoining cycle $C_{\overline{u}_1}$ of the cycle $C_{u_1,u_2,\hdots,u_p}$ overlaps with the first adjoining cycle $C_{v_1}$ of the cycle $C_{v_1,v_2,\hdots,v_q}$ and the adjoining cycles $C_{v_q}$ and $C_{u_1}$ are consecutive then by previous proposition the order of the intersection $C_{u_1,u_2,\hdots,u_p}\bigcap C_{v_1,v_2,\hdots,v_q}$ is indeed $\beta_{\overline{u}_1,\overline{u}_2,\hdots,\overline{u}_\sigma}-1$. Similarly if the adjoining cycles $C_{v_q}$ and $C_{u_1}$ are not consecutive then they will have no common edge and the use of proposition \ref{prp2} gives the order of the intersection $C_{u_1,u_2,\hdots,u_p}\bigcap C_{v_1,v_2,\hdots,v_q}$ as $\beta_{\overline{u}_1,\overline{u}_2,\hdots,\overline{u}_\sigma}-2.$ Similar can be done for the remaining cases.
\end{proof}
\begin{Remark}\label{prp4}
The case when there exists a  $t_0<\sigma<p$  such that  $\overline{u}_{t_0-1}\not\rightarrow\overline{u}_{t_0}$ in above proposition i.e., when cycles $C_{\overline{u}_1},C_{\overline{u}_2},\hdots,C_{\overline{u}_{t_0-1}},C_{\overline{u}_{t_0}},\hdots,C_{\overline{u}_\sigma}$ are not amongst $\sigma$ consecutive adjoining cycles of the cycle $C_{u_1,u_2,\hdots,u_p}$, the order of the intersection $C_{u_1,u_2,\hdots,u_p}\bigcap C_{v_1,v_2,\hdots,v_q} $ can be calculated by applying proposition \ref{prp3} on the overlapping portions.
\end{Remark}

\begin{Proposition}\label{prp5}
\em{Let $\mathcal{J}_{2,m}$ be the graph with the edges $E $ as defined in eq. (1) such that $\{u_1,u_2,\hdots,u_p\}\bigcap \{v_1,v_2,\hdots,v_q\}=\phi$ and $p\le q$ then we have

$$\Big|C_{u_1,u_2,\hdots,u_p}\bigcap C_{v_1,v_2,\hdots,v_q}\Big|=\left\{
                                                       \begin{array}{ll}
                                                         1, & {u_p\rightarrow v_1\;\&\; v_q\not\rightarrow u_1} \\
                                                         1, & {u_p\not\rightarrow v_1\;\&\; v_q\rightarrow u_1} \\
                                                         2, & {u_p\rightarrow v_1\;\&\; v_q\rightarrow u_1} \\
                                                         0, & {otherwise.}
                                                       \end{array}
                                                     \right.$$}
\end{Proposition}

\begin{proof}
In this case the adjoining cycles of $C_{u_1,u_2,\hdots,u_p}$ and $C_{v_1,v_2,\hdots,v_q}$ have no common cycle. However, if the adjoining cycle on one of the extreme ends of the cycle $C_{u_1,u_2,\hdots,u_p}$ is consecutive with the adjoining cycles on one of the extreme ends of the other cycle $C_{v_1,v_2,\hdots,v_q}$ then the intersection $C_{u_1,u_2,\hdots,u_p}\bigcap C_{v_1,v_2,\hdots,v_q}$ will have only one edge. The remaining cases are easy to see.
\end{proof}

In the following three propositions we give some characterizations of $\mathcal{J}_{2,m}$. We fix $E(T_{(j_1i_1,j_2i_2,\hdots,j_mi_m)})$, where $j_\alpha\in\{1,2,\hdots,m\}$ and $i_\alpha\in\{1,2,3\}$, as a subset of $E.$
$s(\mathcal{J}_{2,m}).$
\begin{Proposition}\label{prp6}
\em{A subset $E(T_{(j_1i_1,j_2i_2,\hdots,j_mi_m)}$ of $E$ with $j_\alpha i_\alpha\neq j_\alpha 1$ for all $\alpha$ will belong to $s(\mathcal{J}_{2,m})$ if and only if
$$E(T_{(j_1i_1,j_2i_2,\cdots,j_mi_m)})=E\setminus\{e_{1i_1},e_{2i_2},\hdots, e_{mi_m}\}$$}
\end{Proposition}

\begin{proof}
Since $\mathcal{J}_{2,m}$ is a $m$-cycles graph with cycles $C_1,C_2,\hdots,C_m$ having one edge common in each consecutive cycle and $e_{11},e_{21},\hdots,e_{m1}$ as common edges between consecutive cycles. The cutting down process explains we need to remove exactly $m$ edges, keeping the graph connected and no cycles and no isolated edge left in the graph. Therefore, in order to obtain a spanning tree of $\mathcal{J}_{2,m}$ with none of common edges $e_{11},e_{21},\hdots,e_{m1}$ to be removed, we need to remove exactly one edge from the non common edges from each cycle. This explains the proof of the proposition.


\end{proof}
\begin{Proposition}\label{prp7}
\em{A subset $E(T_{(j_1i_1,j_2i_2,\hdots,j_mi_m)}$ of $E$ with $j_\alpha i_\alpha= j_\alpha 1$ for any $\alpha$ will belong to $s(\mathcal{J}_{2,m})$ if and only if
$$E(T_{(j_1i_1,j_2i_2,\hdots,j_mi_m)})=E\setminus\{e_{j_1i_1},e_{j_2i_2},\hdots, e_{j_mi_m}\}$$
 where, $\{e_{j_1i_1},e_{j_2i_2},\hdots, e_{j_mi_m}\}$ will contain exactly one edge from $C_{(j_{\alpha}-1)(j_{\alpha})}\setminus \{ e_{(j_{\alpha}-1)1},\\ e_{(j_{\alpha}+1)1} \}$ other than $e_{j_\alpha 1}$.}
\end{Proposition}

\begin{proof}
For a spanning tree of $\mathcal{J}_{2,m}$ such that exactly one common edge $e_{j_\alpha 1}$ is removed, we need to remove precisely $m-1$ edges from the remaining edges using the cutting down process. However, we cannot remove more than one edge from the non common edges of the cycle $C_{{(j_{\alpha}-1)1} e_{(j_{\alpha}+1)1}}$ (since this will result a disconnected graph. This explains the proof of the above case.

\end{proof}

\begin{Proposition}\label{prp8}
\em{A subset $E(T_{(j_1i_1,j_2i_2,\hdots,j_mi_m)})\subset E$, where $j_{\alpha} i_{\alpha}= j_{\alpha} 1$ for  $\alpha \in \{r_1,r_2,\hdots,r_{\rho}\}\subset\{1,2,\hdots,m\}$, will belong to $s(\mathcal{J}_{2,m})$ if and only if it satisfies any of the following:
        \begin{enumerate}
            \item\label{8.1} if $e_{j_{r_1}1},e_{j_{(r_2)}1},\hdots, e_{j_{r_{\rho}}1}$ are common edges from consecutive cycles then $$E(T_{(j_1i_1,j_2i_2,\cdots,j_mi_m)})=E\setminus\{e_{j_1i_1},e_{j_2i_2},\hdots, e_{j_mi_m}\}$$ such that $\{e_{j_1i_1},e_{j_2i_2},\hdots, e_{j_mi_m}\}$ will contain exactly exactly one edge from $C_{j_{r_0}j_{r_1}\hdots j_{r_{\rho}}}$ other than $e_{j_{r_1}1},e_{j_{(r_2)}1},\hdots, e_{j_{r_{\rho}}1}$, where $j_{r_0}\rightarrow j_{r_1}$.
            \item\label{8.2} if none of $e_{j_{r_1}1},e_{j_{(r_2)}1},\hdots, e_{j_{r_\rho}1}$ are common edges from consecutive cycles then $$E(T_{(j_1i_1,j_2i_2,\cdots,j_mi_m)})=E\setminus\{e_{j_1i_1},e_{j_2i_2},\hdots, e_{j_mi_m}\}$$ such that for each edge $e_{j_{r_t}1}$ proposition \ref{prp5} holds.
            \item\label{8.3} if some of $e_{j_{r_1}1},e_{j_{(r_2)}1},\hdots, e_{j_{r_\rho}1}$ are common edges from consecutive cycles then $$E(T_{(j_1i_1,j_2i_2,\cdots,j_mi_m)})=E\setminus\{e_{j_1i_1},e_{j_2i_2},\hdots, e_{j_mi_m}\}$$ such that proposition \ref{prp8}.\ref{8.1} is satisfied for the common edges of consecutive cycles and proposition \ref{prp8}.\ref{8.2} is satisfied for remaining common edges.
        \end{enumerate}}

%
\end{Proposition}
\begin{proof}
For the case \ref{8.1}, we need to obtain a spanning tree of $\mathcal{J}_{2,m}$ such that $|r_{\rho}-r_1|_m$ common edges must be removed from $\rho$ consecutive cycles $C_{j_{r_1}},C_{j_{r_2}},\hdots,C_{j_{r_\rho}}$. The remaining $m-|r_\rho-r_1|_m$ edges must be removed in such a way that exactly one edge is removed from the non common edges of the adjoining cycles $C_{j_{r_0}},C_{j_{r_1}},\hdots ,C_{j_{r_\rho}}$ and the remaining $m-|r_\rho-r_1|_m$ cycles of the graph $\mathcal{J}_{2,m}$. This concludes the case.

The remaining cases of the proposition can be visualised in similar manner using the propositions \ref{prp6} and \ref{prp7}. This completes the proof.
\end{proof}

\begin{Remark}\label{rmk9}
If we denote the disjoint classes of subsets of $E$ discussed in propositions \ref{prp6},\ref{prp7} and \ref{prp8} by $\mathcal{C_J}_{1},\mathcal{C_J}_{2},\mathcal{C_J}_{3a},\mathcal{C_J}_{3b},\mathcal{C_J}_{3c}$ respectively, then, we can write $s(\mathcal{J}_{2,m})$ as follows:
$$ s(\mathcal{J}_{2,m})=\mathcal{C_J}_{1}\bigcup\mathcal{C_J}_{2}\bigcup\mathcal{C_J}_{3a}\bigcup\mathcal{C_J}_{3b}\bigcup\mathcal{C_J}_{3c}$$
\end{Remark}
In our next result, we give an important characterization of the $f$-vectors of $\Delta_s(\mathcal{J}_{2,m})$.

\begin{Proposition}\label{fsc}
  \em{Let $\Delta_s(\mathcal{J}_{2,m})$ be a spanning simplicial complex of the graph $\mathcal{J}_{2,m}$, then the $dim(\Delta_s(\mathcal{J}_{2,m}))=2m-1$ with $f-$vector $f(\Delta_s(\mathcal{J}_{2,m}))=(f_0,f_1,\cdots,f_{2m-1})$ and\\
$  f_i=\left(
       \begin{array}{c}
         3m \\
         i+1 \\
       \end{array}
     \right)+\sum\limits_{t=1}^{\tau}(-1)^t
\left[
                             \begin{array}{c}
                               {\sum\limits_{\{ i_{1},i_2,\hdots,i_t\}\in C_{I}^t} \left(
                                       \begin{array}{c}
                                         3m-\sum\limits_{s=1}^{t} \beta_{i_{s}}+\sum\limits_{\{i_u,i_v\}\subseteq\{i_p\}_{p=1}^{t}}\big|C_{i_{u}}\bigcap C_{i_{v}}\big| \\
                                         i+1-\sum\limits_{s=1}^{t} \beta_{i_{s}}+\sum\limits_{\{i_u,i_v\}\subseteq\{i_p\}_{p=1}^{t}}\big|C_{i_{u}}\bigcap C_{i_{v}}\big|\\
                                       \end{array}
                                     \right)}\\
                             \end{array}
                           \right]$
\\where $0\le i\le 2m-1$\\
$I=\{{i_1i_2\hdots i_k}|i_j\in\{1,2,\hdots,m\}\; and \; 1 \le k\le m\\
\;such\;that\;i_{j+1}=i_j+1\; if\; i_j\neq m\; and\; i_{j+1}=1\; if\; i_j= m \}$ and
\\$C_{I}^t=\{Subsets\;of\;I\;of\;cardinality\;t \}.$}
\end{Proposition}

\begin{proof}
  Let $E$ be the edge set of $\mathcal{J}_{2,m}$ and $\mathcal{C_J}_{1},\mathcal{C_J}_{2},\mathcal{C_J}_{3a},\mathcal{C_J}_{3b},\mathcal{C_J}_{3c}$ are disjoint classes of spanning trees of $\mathcal{J}_{2,m}$ then from propositions \ref{prp6}, \ref{prp7}, \ref{prp8} and the remark \ref{rmk9} we have

$$ s(\mathcal{J}_{2,m})=\mathcal{C_J}_{1}\bigcup\mathcal{C_J}_{2}\bigcup\mathcal{C_J}_{3a}\bigcup\mathcal{C_J}_{3b}\bigcup\mathcal{C_J}_{3c}$$

Therefore, by definition \ref{ssc} we can write\\
$ \Delta_s (\mathcal{J}_{2,m})=\Big\langle\mathcal{C_J}_{1}\bigcup\mathcal{C_J}_{2}\bigcup\mathcal{C_J}_{3a}\bigcup\mathcal{C_J}_{3b}\bigcup\mathcal{C_J}_{3c}\Big\rangle$
Since each facet $\hat{E}_{(j_1i_1,j_2i_2,\hdots,j_mi_m)}=E(T_{(j_1i_1,j_2i_2,\hdots,j_mi_m)})$ is obtained by deleting exactly $m$ edges from the edge set of $\mathcal{J}_{2,m}$, keeping in view the propositions \ref{prp6}, \ref{prp7} and \ref{prp8}, therefore dimension of each facet is same i.e., $2m-1$ ( since $ |\hat{E}_{(j_1i_1,j_2i_2,\hdots,j_mi_m)}|=2m$ ) and hence dimension of $\Delta_s(\mathcal{J}_{2,m})$ will be $2m-1$.\\
Also it is clear from the definition of $\Delta_s(\mathcal{J}_{2,m})$ that it contains all those subsets of $E$ which do not contain the given sets of cycles  $\{e_{k1},e_{k2},e_{k3},e_{(k+1)1}\}$ for $k\in\{1,2,\hdots,m-1\}$ and $\{e_{m1},e_{m2},e_{m3},e_{11}\}$  in graph as well as any other cycle in the graph $\mathcal{J}_{2,m}$.\\

Now by lemma \ref{lema1} the total cycles in the graph $\mathcal{J}_{2,m}$ are $$C_{i_1,i_2,\hdots,i_k}\;\;\;i_j\in\{1,2,\hdots,m\}\; and \; 1 \le k\le m,$$
such that $i_{j+1}=i_j+1$ if $i_j\neq m$ and $i_{j+1}=1$ if $i_j= m$, and their total number is $\tau$. Let $F$ be any subset of $E$ of order $i+1$ such that it does not contain any $C_{i_1,i_2,\hdots,i_k}\;\;\;i_j\in\{1,2,\hdots,m\}\; and \; 1 \le k\le m$, in it. The total number of such $F$ is indeed $f_i$ for $0\le i\le 2m-1$. We use inclusion exclusion principle to find this number. Therefore,\\

$f_i=$ Total number of subsets of $E$ of order $i+1$ not containing $C_{i_1,i_2,\hdots,i_k}\;\;\;i_j\in\{1,2,\hdots,m\}\; and \; 1 \le k\le m$ such that $i_{j+1}=i_j+1$ if $i_j\neq m$ and $i_{j+1}=1$ if $i_j= m$.\\

Therefore, using these notations and applying Inclusion Exclusion Principle we can write,\\

$f_i=\Big( $ Total number of subsets of $E$ of order $i+1\Big)- \sum\limits_{\{ i_{1}\}\in C_{I}^{1}}\Big($ subset of $E$  of order $i+1$ containing $C_{i_s}$ for $s=1\Big)+\sum\limits_{\{ i_{1},i_2\}\in C_{I}^{2}}\Big($ subset of $E$ of order $i+1$ containing both $C_{i_{s}}$ for both $1\le s\le2 \Big)- \cdots +(-1)^{\tau}\sum\limits_{\{ i_{1},i_2,\hdots,i_{\tau}\}\in C_{I}^{\tau}}\Big($subset of $E$ of order $i+1$ simultaneously containing each $C_{i_{s}}$ for all $1\le s\le\tau\Big)$
\\
\\This implies\\
$f_i=\left(
       \begin{array}{c}
         3m \\
         i+1 \\
       \end{array}
     \right)-\Big[\begin{array}{c}
                  \sum\limits_{\{ i_{1}\}\in C_{I}^{1}}
                                       \left(
                                       \begin{array}{c}
                                         3m-\beta_{i_{1}}\\
                                         i+1-\beta_{i_{1}}\\
                                       \end{array}
                                     \right) \\
                \end{array}
     \Big]+\\ \left[
                \begin{array}{c}
                  \sum\limits_{\{ i_{1},i_2\}\in C_{I}^{2}}
                                       \left(
                                       \begin{array}{c}
                                         3m-\sum\limits_{s=1}^{2} \beta_{i_{s}}+\sum\limits_{\{i_u,i_v\}\subseteq\{i_p\}_{p=1}^{2}}\big|C_{i_{u}}\bigcap C_{i_{v}}\big| \\
                                         i+1-\sum\limits_{s=1}^{2} \beta_{i_{s}}+\sum\limits_{\{i_u,i_v\}\subseteq\{i_p\}_{p=1}^{2}}\big|C_{i_{u}}\bigcap C_{i_{v}}\big| \\
                                       \end{array}
                                     \right) \\
                \end{array}
              \right]\\
 - \cdots+(-1)^{\tau}\\
\left[
                             \begin{array}{c}
                               {\sum\limits_{\{ i_{1},i_2,\hdots,i_{\tau}\}\in C_{I}^{\tau}} \left(
                                       \begin{array}{c}
                                         3m-\sum\limits_{s=1}^{\tau} \beta_{i_{s}}+\sum\limits_{\{i_u,i_v\}\subseteq\{i_p\}_{p=1}^{\tau}}\big|C_{i_{u}}\bigcap C_{i_{v}}\big| \\
                                         i+1-\sum\limits_{s=1}^{\tau} \beta_{i_{s}}+\sum\limits_{\{i_u,i_v\}\subseteq\{i_p\}_{p=1}^{\tau}}\big|C_{i_{u}}\bigcap C_{i_{v}}\big| \\
                                       \end{array}
                                     \right)}\\
                             \end{array}
                           \right]$
\\This implies\\
$  f_i=\left(
       \begin{array}{c}
         3m \\
         i+1 \\
       \end{array}
     \right)+\sum\limits_{t=1}^{\tau}(-1)^t
\left[
                             \begin{array}{c}
                               {\sum\limits_{\{ i_{1},i_2,\hdots,i_t\}\in C_{I}^t} \left(
                                       \begin{array}{c}
                                         3m-\sum\limits_{s=1}^{t} \beta_{i_{s}}+\sum\limits_{\{i_u,i_v\}\subseteq\{i_p\}_{p=1}^{t}}\big|C_{i_{u}}\bigcap C_{i_{v}}\big| \\
                                         i+1-\sum\limits_{s=1}^{t} \beta_{i_{s}}+\sum\limits_{\{i_u,i_v\}\subseteq\{i_p\}_{p=1}^{t}}\big|C_{i_{u}}\bigcap C_{i_{v}}\big|\\
                                       \end{array}
                                     \right)}\\
                             \end{array}
                           \right]$

\end{proof}

\begin{cor}
Let $\Delta_s (\mathcal{J}_{2,3})$ be a spanning simplicial complex of the Jahangir's graph $\mathcal{J}_{2,m}$ given in Figure 1, then the $dim(\Delta_s (\mathcal{J}_{2,3}))=5$ and $\tau=3^2=9$. Therefore, $f-$vectors $f(\Delta_s (\mathcal{J}_{2,3}))=(f_0,f_1,\hdots,f_5)$ and
\\
$f_i=\left(
       \begin{array}{c}
         9 \\
         i+1 \\
       \end{array}
     \right)-\Big[\begin{array}{c}
                  \sum\limits_{\{ i_{1}\}\in C_{I}^{1}}
                                       \left(
                                       \begin{array}{c}
                                         9-\beta_{i_{1}}\\
                                         i+1-\beta_{i_{1}}\\
                                       \end{array}
                                     \right) \\
                \end{array}
     \Big]+\\ \left[
                \begin{array}{c}
                  \sum\limits_{\{ i_{1},i_2\}\in C_{I}^{2}}
                                       \left(
                                       \begin{array}{c}
                                         9-\sum\limits_{s=1}^{2} \beta_{i_{s}}+\sum\limits_{\{i_u,i_v\}\subseteq\{i_p\}_{p=1}^{2}}\big|C_{i_{u}}\bigcap C_{i_{v}}\big| \\
                                         i+1-\sum\limits_{s=1}^{2} \beta_{i_{s}}+\sum\limits_{\{i_u,i_v\}\subseteq\{i_p\}_{p=1}^{2}}\big|C_{i_{u}}\bigcap C_{i_{v}}\big| \\
                                       \end{array}
                                     \right) \\
                \end{array}
              \right]\\
 - \cdots+(-1)^{9}\\
\left[
                             \begin{array}{c}
                               {\sum\limits_{\{ i_{1},i_2,\hdots,i_{9}\}\in C_{I}^{9}} \left(
                                       \begin{array}{c}
                                         3m-\sum\limits_{s=1}^{9} \beta_{i_{s}}+\sum\limits_{\{i_u,i_v\}\subseteq\{i_p\}_{p=1}^{9}}\big|C_{i_{u}}\bigcap C_{i_{v}}\big| \\
                                         i+1-\sum\limits_{s=1}^{9} \beta_{i_{s}}+\sum\limits_{\{i_u,i_v\}\subseteq\{i_p\}_{p=1}^{9}}\big|C_{i_{u}}\bigcap C_{i_{v}}\big| \\
                                       \end{array}
                                     \right)}\\
                             \end{array}
                           \right]$
\\where $0\le i\le 5.$
\end{cor}

For a simplicial complex $\Delta$ over $[n]$, one would associate to it the
Stanley-Reisner ideal, that is, the monomial ideal $I_{\mathcal N}(\Delta)$ in
$S=k[x_1, x_2,\ldots ,x_n]$ generated by monomials corresponding to
non-faces of this complex (here we are assigning one variable of the
polynomial ring to each vertex of the complex). It is well known
that the Face ring $k[\Delta]=S/I_{\mathcal N}(\Delta)$
is a standard graded algebra. We refer the readers to \cite{HP} and
\cite{Vi} for more details about graded algebra $A$, the Hilbert
function $H(A,t)$ and the Hilbert series $H_t(A)$ of a graded algebra.

Our main result of this section is as follows;
\begin{Theorem}\label{Hil} {\em Let $\Delta_s(\mathcal{J}_{2,m}) $ be the spanning simplicial complex of
$\mathcal{J}_{2,m}$, then the Hilbert series of the Face ring
$k\big[\Delta_s(\mathcal{J}_{2,m})\big]$ is given by,\\
$H(k[\Delta_s(\mathcal{J}_{2,m})],t)=1+\sum\limits_{i=0}^{d}\frac{{n\choose
{i+1}}{t^{i+1}}}{(1-t)^{i+1}}+\sum\limits_{i=0}^{d}\sum\limits_{k=1}^{\tau}(-1)^k\\
\tiny{
\left[
                             \begin{array}{c}
                               {\sum\limits_{\{ i_{1},i_2,\hdots,i_k\}\in C_{I}^k} \left(
                                       \begin{array}{c}
                                         3m-\sum\limits_{s=1}^{k} \beta_{i_{s}}+\sum\limits_{\{i_u,i_v\}\subseteq\{i_p\}_{p=1}^{k}}\big|C_{i_{u}}\bigcap C_{i_{v}}\big| \\
                                         i+1-\sum\limits_{s=1}^{k} \beta_{i_{s}}+\sum\limits_{\{i_u,i_v\}\subseteq\{i_p\}_{p=1}^{k}}\big|C_{i_{u}}\bigcap C_{i_{v}}\big|\\
                                       \end{array}
                                     \right)}\\
                             \end{array}
                           \right]} \frac{t^{i+1}}{(1-t)^{i+1}}$}
\end{Theorem}

\begin{proof}
From \cite{Vi}, we know that if $\Delta$ is a simplicial complex of
dimension $d$ and $f(\Delta)=(f_0, f_1, \ldots,f_d)$ its $f$-vector,
then the Hilbert series of the face ring $k[\Delta]$ is given
by $$H(k[\Delta],t)= 1+\sum_{i=0}^{d}\frac{f_i
t^{i+1}}{(1-t)^{i+1}}.$$ By substituting the values of $f_i$'s from
Proposition \ref{fsc} in this above expression, we get the desired result.
\end{proof}

\section{Cohen-Macaulayness of the face ring of $\Delta_s(\mathcal{J}_{2,m})$}
In this section, we present the Cohen-Macaulayness of the face ring of SSC $\Delta_s(\mathcal{J}_{2,m})$, using the notions and results from \cite{AR}.
\begin{Definition}\label{qlq}\cite{AR}\\
\em{Let $I\subset S=k[x_1,x_2,\hdots,x_n]$ be a monomial ideal, we say that $I$ will have the \textit{qausi-linear quotients}, if there exists a minimal monomial system of generators $m_1,m_2,\hdots,m_r$ such that $mindeg(\hat{I}_{m_i})=1$ for all $1<i\le r$, where
$$\hat{I}_{m_i}=(m_1,m_2,\hdots,m_{i-1}):(m_i).$$}
\end{Definition}
\begin{Theorem}\cite{AR}
\em{Let $\Delta$ be a pure simplicial complex of dimension $d$ over $[n]$. Then $\Delta$ will be a shellable simplicial complex if and only if $I_{\mathcal{F}}(\Delta)$ will have the qausi-linear quotients.}
\end{Theorem}
\begin{cor}\label{frcm}\cite{AR}
\em{The face ring of a pure simplicial complex $\Delta$ over $[n]$ is Cohen Macaulay if and only if $I_{\mathcal{F}}(\Delta)$ has quasi-linear quotients.}
\end{cor}

Here, we present the main result of this section.
\begin{Theorem}\label{CM}
\em{The face ring of $\Delta_s(\mathcal{J}_{2,m})$ is
Cohen-Macaulay.}
\end{Theorem}
\begin{proof}
By corollary \ref{frcm}, it is sufficient to show that $I_{\mathcal{F}}\big(\Delta_s(\mathcal{J}_{2,m})\big)$ has a quasi-linear quotients in $S=k[x_{11},x_{12},x_{13},x_{21},x_{22},x_{23},x_{31},\hdots,x_{m1},x_{m2},x_{m3}]$. By propositions \ref{prp6}, \ref{prp7}, \ref{prp8} and the remark \ref{rmk9}, we have
$$s(\mathcal{J}_{2,m})=\mathcal{C_J}_{1}\bigcup\mathcal{C_J}_{2}\bigcup\mathcal{C_J}_{3a}\bigcup\mathcal{C_J}_{3b}\bigcup\mathcal{C_J}_{3c}$$
Therefore,
$$ \Delta_s (\mathcal{J}_{2,m})=\Big\langle\hat{E}_{(j_1i_1,j_2i_2,\hdots,j_mi_m)}=E\backslash \{e_{j_1i_1},e_{j_2i_2},\hdots,e_{j_mi_m}\}\mid \hat{E}_{(j_1i_1,j_2i_2,\hdots,j_mi_m)}\in s (\mathcal{J}_{2,m})\Big\rangle$$
and hence we can write,
$$ I_{\mathcal{F}}(\Delta_s(\mathcal{J}_{2,m}))=\Big( x_{\hat{E}_{(j_1i_1,j_2i_2,\hdots,j_mi_m)}}\mid \hat{E}_{(j_1i_1,j_2i_2,\hdots,j_mi_m)}\in s (\mathcal{J}_{2,m})\Big).$$
Here $ I_{\mathcal{F}}(\Delta_s(\mathcal{J}_{2,m}))$ is a pure monomial ideal of degree $2m-1$ with $x_{\hat{E}_{(j_1i_1,j_2i_2,\hdots,j_mi_m)}}$ as the product of all variables in $S$ except $x_{j_1i_1},x_{j_2i_2},\hdots,x_{j_mi_m}$. Now we will show that $ I_{\mathcal{F}}(\Delta_s(\mathcal{J}_{2,m}))$ has qausi-linear quotients with respect to the following generating system:\\
$\{x_{\hat{E}_{(11,21,\hdots,(m-1)1,j_mi_m)}}\mid i_r\neq 1\},\{x_{\hat{E}_{(11,21,\hdots,(m-2)1,(m-1)i_{m-1},j_mi_m)}}\mid i_{m-1}\neq 1\},\\
\{x_{\hat{E}_{(11,21,\hdots,(m-3)1,(m-2)i_{m-2},j_{m-1}i_{m-1},j_mi_m)}}\mid i_{m-2}\neq 1\},\hdots ,\{x_{\hat{E}_{(11,2i_2,j_3i_3\hdots,j_mi_m)}}\mid i_2\neq 1\},\\
\{x_{\hat{E}_{(1i_1,j_2i_2,\hdots,j_mi_m)}}\mid i_1\neq 1\}$\\
Let us put
\\$\begin{array}{c}
    C_{(11,21,\hdots,(m-2)1,(m-1)i_{m-1},j_mi_m)}=\{x_{\hat{E}_{(11,21,\hdots,(m-2)1,(m-1)i_{m-1},j_mi_m)}}\mid i_{m-1}\neq 1\}, \\
    C_{(11,21,\hdots,(m-3)1,(m-2)i_{m-2},j_{m-1}i_{m-1}j_mi_m)}=\{x_{\hat{E}_{(11,21,\hdots,(m-3)1,(m-2)i_{m-2},j_{m-1}i_{m-1},j_mi_m)}}\mid i_{m-1}\neq 1\}, \\
  \vdots \\
   C_{(1i_1,j_2i_2,\hdots,j_mi_m)}=\{x_{\hat{E}_{(1i_1,j_2i_2,\hdots,j_mi_m)}}\mid i_1\neq 1\}.
\end{array}$\\
 Also for any $C_{(j_1i_1,j_2i_2,\hdots,j_mi_m)}$, denote $\bar{C}_{(j_1i_1,j_2i_2,\hdots,j_mi_m)}$  as the residue collection of all the generators which precedes $C_{(j_1i_1,j_2i_2,\hdots,j_mi_m)}$ in the above order. We will show that
$$(\bar{C}_{(j_1i_1,j_2i_2,\hdots,j_mi_m)}):( x_{\hat{E}_{(j_1i_1,j_2i_2,\hdots,j_mi_m)}})$$
contains atleast one linear generator.\\
Now for any generator $x_{\hat{E}_{(11,\hdots,(k-1)1,j_ki_k,\hdots,j_mi_m)}}$, the above said system of generators guarantee the existence of a generator $x_{\hat{E}_{(11,\hdots,(k-1)1,j_\alpha i_\alpha,j_{k+1}i_{k+1},\hdots,j_mi_m)}}$ in $\bar{C}_{(11,\hdots,(k-1)1,j_ki_k,\hdots,j_mi_m)}$ such that $j_\alpha i_\alpha\neq j_ki_k$. Therefore, by using the definition of colon ideal it is easy to see that
$$(\bar{C}_{(11,\hdots,(k-1)1,j_ki_k,\hdots,j_mi_m)}):(x_{(11,\hdots,(k-1)1,j_ki_k,\hdots,j_mi_m)})$$
contains a linear generator $x_{j_ki_k}$. Hence $I_{\mathcal{F}}(\Delta_s(\mathcal{J}_{2,m}))$ has quasi-linear quotients, as required.

\end{proof}

\section{Conclusions and Scopes}

We conclude this paper with some open scopes aswell as some constraints related to our work.
\begin{itemize}
  \item The results given in this paper can be naturally extended for any integer $n\ge2$.
  \item The scope of SSC of a graph can be explored for some other classes of graphs like the wheel graph $W_n$ etc. However, since finding spanning trees of a general graph is a NP-hard problem, therefore the results given here are not easily extendable for a general class of graph.
  \item In view of the work done in \cite{IB1, IB2}, we intend to find some scopes of the SSC in studying sensor networks.
\end{itemize}

%
%

\end{document}